\documentclass[preprint,12pt]{elsarticle}
\usepackage{amsfonts}
\usepackage{amsthm}
\usepackage{amsmath}
\usepackage[all]{xy}
\usepackage{titlesec}
\usepackage{dsfont}
\usepackage{amssymb}
\usepackage{extarrows}
\usepackage{mathrsfs}
\usepackage{arydshln}
\usepackage{multirow}
\usepackage{color}
\usepackage{pb-diagram}
\usepackage{times,epic,eepic,pb-diagram,tikz-cd}
\usepackage{graphicx}
\graphicspath{{img/}}
\usetikzlibrary{matrix,arrows,trees}
\usetikzlibrary{decorations.markings}
\usetikzlibrary{decorations.pathmorphing}

\usepackage[bookmarksnumbered, colorlinks, plainpages]{hyperref}
in \textwidth 6 true in

\journal{xx}
\newtheorem{theorem}{Theorem}[section]
\newtheorem{proposition}[theorem]{Proposition}
\newtheorem{lemma}[theorem]{Lemma}
\newtheorem{corollary}[theorem]{Corollary}

\newtheorem{remark}[theorem]{Remark}
\newtheorem*{criterion theorem}{Criterion Theorem}

\newcommand{\Z}{\ensuremath{\mathbb{Z}}}
\newcommand{\fhe}[0]{\ensuremath{{\scriptstyle\circ}}}
\newcommand{\Pn}{\ensuremath{P^{2k}(n)}}

\begin{document}

\begin{frontmatter}

\title{Homotopy  classification of   $S^{2k-1}$-bundles over $S^{2k}$}

\author[1]{Zhongjian Zhu}
\ead{20160118@wzu.edu.cn}
\author[2]{Jianzhong Pan\corref{cor}}
\ead{pjz@amss.ac.cn}
\cortext[cor]{Corresponding author}
\address[1]{College of Mathematics and Physics, Wenzhou University, Wenzhou 325035, China}
\address[2]{Hua Loo-Keng Key Mathematical Laboratory, Institute of Mathematics,\\ Academy of Mathematics and Systems Science,
	Chinese Academy of Sciences; \\University of Chinese Academy of Sciences, Beijing,  100190, China}

\begin{abstract}
	In this paper,  we  classify the homotopy types of the total spaces of $S^{2k-1}$-bundles (or fibrations)  over $S^{2k}$ for  $2\leq  k\leq 6$. One of the  key ingredients in the argument  is the  new  necessary and sufficient conditions for a CW complex to be  homotopy equivalent to the total space of a sphere bundle (fibration).  When $k=4$, the classification results  provide a negative answer to  the conjecture in \cite{S8bundles}.  
\end{abstract}

\end{frontmatter}

\section{Introduction}
\label{intro}
 This is the second paper in a series that we are devoting to the homotopy theory of sphere fibrations. 
Motivated by  Kitchloo and Shankar's criterion in \cite[Theorem 1]{S3S4bundle} for determining whether a given CW complex is homotopy equivalent to  $S^{3}$-fibration over $S^{4}$, in our previous work \cite{ZhuPanfiber}, we gave necessary and sufficient conditions for a CW complex to have the homotopy type of the total space of an $S^{2k-1}$-fibration over $S^{2k}$ for any $k\geq 2$. There is the canonical homotopy cofibration 
\[S^{k} \xrightarrow{[n]} S^{k} \xrightarrow{i_k} P^{k+1}(n)\xrightarrow{p_{k+1}} S^{k+1},\] 
where $[n]$  denotes the map of degree $n$.
It induces an exact sequence of the homotopy groups for the pair $(\Pn, S^{2k-1})$
\begin{align}
	\xymatrix{
		\pi_{4k-2}(S^{2k-1})	 \ar[r]^-{i_{2k-1\ast}}  &	\pi_{4k-2}(\Pn)	 \ar[r]^-{j_\ast}  & \pi_{4k-2}(\Pn,S^{2k-1})\ar[r]^-{\partial} & \pi_{4k-3}(S^{2k-1}).  } \label{exact1 pi(P S)}
\end{align}

Let $K^n_k=Ker(p_{2k\ast}:\pi_{4k-2}(\Pn)\to \pi_{4k-2}(S^{2k}))$.

By Lemma (1) of \cite{SASAO2},  (\ref{exact1 pi(P S)}) deduces the short exact sequence
\begin{align}
	0 \to  \pi_{4k-2}( S^{2k-1})/n \pi_{4k-2}( S^{2k-1}) \overset{i_{2k-1\ast}}{\to}  K^n_k \overset{j_\ast}{\to} j_{\ast}(K^n_k) \to 0 \label{exact: K}
\end{align}
with  $ j_{\ast}(K^n_k)$ a cyclic group generated by  $[X_{2k},\iota_{2k-1}]_r$ for $2|n$ or $k=2,4$ by \cite[Lemma 2.2]{ZhuPanfiber}, where $X_{2k}$ is  a fixed  generator of $\pi_{2k}(\Pn, S^{2k-1})=\Z\{X_{2k}\}$ and $[X_{2k},\iota_{2k-1}]_r$ is the relative Whitehead product (defined in \cite{Blakers}) of  $X_{2k}$  and $\iota_{2k-1}$, the homotopy class of the identity map on the sphere $S^{2k-1}$. We denote the composition $f\fhe g$ of two maps simply by $fg$.  \cite[Corollay 1.2 and Theorem 1.4]{ZhuPanfiber} gives the following  theorem
\begin{criterion theorem}\label{Criterion theorem}
		Let $n\geq 2$ and  $\theta_{k}^n\in\pi_{4k-2}(\Pn)$  be any fixed lift of $[X_{2k},\iota_{2k-1}]_r$ by the map  $j_{\ast}$.  Then a CW-complex $X$ is homotopy equivalent to the total space of an $S^{2k-1}$-fibration over $S^{2k}$   if and only if  
		\begin{itemize}
			\item $2\mid n$ if $k\ne 2,4$ and 
			\item $X\simeq \Pn\cup_{f}e^{4k-1}$, such that   $f=\theta_{k}^n + i_{2k-1} \gamma$ with $\gamma\in \pi_{4k-2}(S^{2k-1})$. 
		\end{itemize}
\end{criterion theorem}
 
Note that for $n=1$, any such $X=\Pn\cup_{f}e^{4k-1}$ is homotopy equivalent to $S^{4k-1}$.  For $n=0$,   $X=\Pn\cup_{f}e^{4k-1}=(S^{2k-1}\vee S^{2k})\cup_{f}e^{4k-1}$ and the corresponding sphere fibration has a section.  The  homotopy classification of  the total spaces of such sphere fibrations was considered by Sasao  \cite{SASAO3}.   In this paper we classify  up to  homotopy the total spaces of  $S^{2k-1}$-fibrations over $S^{2k}$ for  $2\leq k\leq 6$ and $n\geq 2$.

 To state our main results, we need the following notations. 
 
Let $G_k^n$ be the number of the homotopy types of  the  total spaces of  $S^{2k-1}$-fibrations over $S^{2k}$.

 $n=2^rp_1^{e_1}p_2^{e_2}\cdots p_s^{e_s}$ denotes the prime factorization of $n$. Here $r\geq 0$, $p_1<p_2<\cdots< p_s$ are odd primes, and $e_1,e_2,\cdots, e_s$ are positive integers. If $r=0$ or $1$ and  $p_i\equiv 1$ (mod $4$) for all $i=1,\cdots, s$, then we say that $n$ satisfies $\bigstar$.

  In order to list the homotopy types of the total spaces of $S^{2k-1}$-fibrations over $S^{2k}$, the following homotopy groups $\pi_{4k-2}( S^{2k-1})$   are needed:
\newline
$\pi_{6}(S^3)\cong \Z_{12}$, $\pi_{14}(S^7)\cong\Z_{120}$,   $ \pi_{10}(S^5)=\Z_{2}\{\nu_5\eta_8^2\}$;  $ \pi_{18}(S^9)=\Z_{2}\{\nu_9^3\} \oplus \Z_{2}\{\mu_9\}\oplus \Z_{2}\{\eta_9\varepsilon_{10}\}\oplus \Z_{2}\{\sigma_9\eta^2_{16}\}$ and 
$ \pi_{22}(S^{11})=\Z_{8}\{\zeta_{11}\} \oplus \Z_{9}\{\alpha_{3}^{11}\}\oplus \Z_{7}\{\alpha_{1}^{11}\}$,
 where $\alpha_{3}^{11}:=\Sigma^8 \alpha_{3}(3)$ and   $\alpha_{1}^{11}:=\Sigma^8 \alpha_{1}(3)$ are the generators of the corresponding odd primary components of $\pi_{4k-2}(S^{2k-1})$ and the others are the generators of the $2$-primary components given in \cite{Toda}. 
 
For an integer $s$, let $\rho_s^n=1$ or $0$ according as $s|n$ or $s\nmid n$.  Let $S_0=\Z_{2}\{\nu_9^3\} \oplus \Z_{2}\{\mu_9\}\oplus \Z_{2}\{\eta_9\varepsilon_{10}\}\subset \pi_{18}(S^9)$.

 \qquad

\textbf{A note on Orientation.}  Let $H_{\ast}(Y)$ denote the integral homology groups of space $Y$. 
Fix a generator $X_{2k}\in \pi_{2k}(\Pn,S^{2k-1})$ and a generator $\iota\in H_{4k-1}(S^{4k-1})$.  If $S^{2k-1}\rightarrow X\xrightarrow{\pi}S^{2k}$ is a sphere fibration, then $X$ is a Poincare duality complex.  By above \textbf{Criterion Theorem}, $X$ has a cell 
decomposition  
$$X=\Pn\cup_{f}e^{4k-1}~\text{and}~  j_{\ast}(f)= [X_{2k},\iota_{2k-1}]_r.$$
Denote by $q:X \to S^{4k-1}$ the natural quotient map.
An orientation class
\begin{align} 
\sigma_f\in H_{4k-1}(X) ~\text{is said to be $\iota$-compatible if}~q_{\ast}(\sigma_f)=\iota.  \label{Def:Orien of X}
\end{align}
 When $k=2, 4$, it is well known that any  $S^{2k-1}$-fibration over $S^{2k}$ is fiber homotopy  equivalent to an $S^{2k-1}$-bundle  over $S^{2k}$.  In this case the total space is a manifold denoted by $M_{m,n}^{4k-1}$ in the following.    By Lemma \ref{Lem:orien of M},   the orientation $\sigma^{4k-1}_{m,n}$ of $M_{m,n}^{4k-1}$,  defined by the orientations of base space $S^{2k}$ and fibre $S^{2k-1}$,  is  $\iota$-compatible for some  $\iota$ independent of $m,n$.

 \qquad

The isomorphism classes of  $S^{2k-1}$-bundles  over $S^{2k}$ for $k=2,4$ are classified by the characteristic maps in  $\pi_{2k-1}(SO(2k))\cong \Z\oplus \Z$.  
The fibration sequence  
$SO(2k-1)\xrightarrow{ \mathbf{i}}SO(2k)\xrightarrow{\mathbf{ p}} S^{2k-1}$  induces a short exact sequence 
\begin{align*}
	0\!\rightarrow\! \Z\{\bar \rho_{2k-1}\}\!=\! \pi_{2k-1}(SO(2k\!-\!1))\xrightarrow{\mathbf{i}_{\ast}}\pi_{2k\!-\!1}(SO(2k))\xrightarrow{\mathbf{p}_{\ast}} \pi_{2k-1}(S^{2k-1})\!=\!\Z\{\iota_{2k-1}\}\!\rightarrow\! 0. 
\end{align*}
with $\pi_{2k-1}(SO(2k))=\Z\{\bar \rho_{2k}\}\oplus \Z\{\bar\sigma_{2k}\}$, where the generators  $\bar \rho_{2k-1}$,  $\bar \rho_{2k}$ and $\bar\sigma_{2k}$ are given by using the quaternion and octonion  multiplication in $S^{3}\subset \mathbb{H}$ and $S^7\subset \mathbb{O}$ \cite[(2.1)]{Tamura}:
\begin{align*}
	&\bar \rho_{2k-1}: S^{2k-1}\rightarrow SO(2k-1)~~\text{with}~ \bar\rho_{2k-1}(x)(y):=xyx^{-1}; \\
	& \bar \rho_{2k}:=\mathbf{i}_{\ast}(\bar \rho_{2k-1});~~~~~~\bar \sigma_{2k}:  S^{2k-1}\rightarrow SO(2k) ~~\text{with}~ \bar \sigma_{2k}(x)(y):=xy.
\end{align*}
 We simplify  the elements  $m\bar \rho_{2k-1}+n\bar\sigma_{2k}$ of $\pi_{2k-1}(SO(2k))$ by  ordered pairs of integers, $[m,n]$.   Denote the   sphere bundle  corresponding to $[m,n ]$ by \[S^{2k-1}\rightarrow M^{4k-1}_{m,n}\xrightarrow{\pi} S^{2k}.\] 

By the surgery theory of manifolds,  Crowley and Escher \cite{Crowley} gave a classification of the total spaces of $S^3$-bundles over $S^4$ up to homotopy equivalence, homeomorphism and diffeomorphism. The diffeomorphism classification of the total spaces of $S^7$-bundles over $S^8$ was obtained by  M. Grey \cite{S8bundles}. Partial classification up to homotopy and homeomorphism of the total spaces of $S^7$-bundles over $S^8$ was also obtained by  M. Grey \cite{S8bundles} and Ajay,Tibor \cite{S7bundlesAjay}  .

 Using the  relationship between the characteristic maps  of  $S^{2k-1}$-bundles over $S^{2k}$ 
and the attaching maps of their total spaces for $k=2,4$ in \cite{JamesII},  we  classify up to  homotopy  the total spaces of such bundles  via homotopy-theoretic methods ( Theorem \ref{thm:S3S4 bundle}, \ref{thm:S7S8 bundle}).  Theorem \ref{thm:S3S4 bundle} agrees with the result  given by Crowley in \cite[Theorem 1.1]{Crowley}.  Theorem \ref{thm:S7S8 bundle}  provides an answer to a conjecture  in  \cite{S8bundles}.

Let  $p^r|| n$ mean $p^r | n$ and $p^{r+1}\nmid n$ for a prime number $p$ and	$(a,b)$ be the greatest common divisor of integers $a$ and $b$.

\begin{theorem}\label{thm:S3S4 bundle}~~Let $n,n'\geq 2$.

\begin{enumerate}[(I)] 
	\item The manifolds \( M^7_{m,n} \) and \( M^7_{m',n'} \) are orientation preserving homotopy equivalent if and only if \( n = n' \) and  $ tm' \equiv  m $ (mod $(n,12)$), where \( t^2 \equiv 1 \) (mod $(n,12)$).
	\item Orientation reversing homotopy equivalences between any \( M^7_{m,n} \) and \( M^7_{m',n} \) can only exist when $n$ satisfies $\bigstar$; hence $3\nmid  n$ and $4\nmid n$.  Furthermore if $2\nmid n$, then  the single oriented homotopy type admits an orientation reversing self homotopy equivalence; if $2||n$, then \( M^7_{m,n} \) is orientation reversing homotopy equivalent to \( M^7_{m',n} \) if and only if \( m + m' \not\equiv 0 \) (mod $2$).
\end{enumerate}
\end{theorem}

\begin{theorem}\label{thm:S7S8 bundle}~~Let $n,n'\geq2$.
	
	\begin{enumerate}[(I)]
		\item The manifolds \( M^{15}_{m,n} \) and \( M^{15}_{m',n'} \) are orientation preserving homotopy equivalent if and only if \( n = n' \) and  
			\begin{enumerate}[(1)]
			\item \label{thm:S7S8 bundle 8 nmid n} $tm\equiv m'$ (mod $(n,120)$) with  $t^2\equiv 1$ ~(mod $(n,120)$) for $8 \nmid n$;
			\item \label{thm:S7S8 bundle 8||n}   $tm\equiv m'$ (mod $(n,120)$) with  $t^2\equiv 1 ~\text{(mod $2n$)}$ or  $tm+60\equiv m'$ (mod $(n,120)$) with  $t^2\equiv 1+n ~\text{(mod $2n$)}$  for $8|| n$;
			\item \label{thm:S7S8 bundle 16|n}  $tm\equiv m'$ (mod $(n,120)$) with  $t^2\equiv 1 ~\text{(mod $n$)}$ for $16| n$.
		\end{enumerate}

		\item Orientation reversing homotopy equivalences between any \( M^{15}_{m,n} \) and \( M^{15}_{m',n} \) can only exist when $n$ satisfies $\bigstar$; hence $3\nmid  n$ and $4\nmid n$. Furthermore if $2\nmid n$ and $5\nmid n$, then  the single oriented homotopy type admits an orientation reversing self homotopy equivalence;  if $2\nmid n$ and $5|n$,  \( M^{15}_{m,n}\) is orientation reversing homotopy equivalent to \( M^{15}_{m',n} \) if and only if   $m' \equiv \pm 2m$  (mod $5$); if $2||n$,    \( M^{15}_{m,n} \) is orientation reversing homotopy equivalent to \( M^{15}_{m',n} \) if and only if  $m' + m \not\equiv 0$~(mod $2$) and  $m' \equiv \pm 2m$  (mod $(5,n)$).
	
	\end{enumerate}
\end{theorem}

 \begin{remark}~~
 	
 		\begin{enumerate}[(I)]
 		\item If $n=0$, then from easy computation by  \cite[Theorem 1.6]{JamesI},  we get $M^{4k-1}_{m,0}\simeq M^{4k-1}_{m',0}$  if and only if  $ m' \equiv \pm m $ (mod $n_{2k-1}$), where $n_{2k-1}=12$ or $120$ according as $k=2$ or $4$.  $M^{4k-1}_{m,0}$   has an orientation reversing 
 		self-diﬀeomorphism by  \cite[Lemma  2.16]{S8bundles} for $k=4$ and the same proof of that, with the dimensions of spheres and disks changed accordingly,   also applies to the case $k=2$.

 	\item 	When $8\nmid n$, (\ref{thm:S7S8 bundle 8 nmid n}) of Theorem \ref{thm:S7S8 bundle} confirms  the Conjecture 5.0.3 of \cite{S8bundles}. However, the conjecture is not true otherwise. For example,
 	
 	when $8||n$, by Theorem \ref{thm:S7S8 bundle} (\ref{thm:S7S8 bundle 8||n}), $M^{15}_{60,120}$ and  $M^{15}_{0,120}$ are orientation preserving homotopy equivalent but  $(n,m,m')=(120,60,0)$ does not  satisfy the conditions of  the Conjecture; 
 	
    when $16|n$,  
 $M^{15}_{1,16}$ and  $M^{15}_{5,16}$ are not orientation preserving homotopy equivalent by Theorem \ref{thm:S7S8 bundle} (\ref{thm:S7S8 bundle 16|n}) but $(n,m,m')=(16,1,5)$  satisfies the condition of  the Conjecture.
 		\end{enumerate}
 \end{remark}

 Theorem \ref{thm:S3S4 bundle} and \ref{thm:S7S8 bundle} imply the following corollary. 
 \begin{corollary}\label{Cor:Gn for k=2,4}  The number of  the total spaces of $S^{2k-1}$-fibrations over $S^{2k}$ up to homotopy for $k=2,4$ are given as follows
 		\begin{itemize}
 		\item   
 		$G_2^n=\left\{
 		\begin{array}{ll}
 			1, &\! \hbox{if $(12,n)=2$ and $n$ satisfies $\bigstar$} \\
 			\frac{(r+1)(t+1)}{2}, &\! \hbox{if $(12,n)=2^rt, r=0,1,2, t=1,3$ and $n$ does not satisfies $\bigstar$.}
 		\end{array}
 		\right. $
 		\item 	$G_4^n=\left\{
 		\begin{array}{ll}
 			t, &~ \hbox{if $(240,n)=t^2+1$ and $t=1,2,3$ and $n$ satisfies $\bigstar$} \\
 			\frac{(r+1)(t_1+1)(t_2+1)}{4}, &\! \hbox{ $\begin{tabular}{c}
 				if	$(240,n)=2^rt_1t_2, r=0,1,2,4, t_1=1,3,$ \\
 					$t_2=1,5$	and $n$ does not satisfies $\bigstar$\\
 				\end{tabular}$}
 		\end{array}
 		\right. $
 	\end{itemize}
 \end{corollary}

 \begin{theorem}\label{Thm: all odd k fibr}
For $k\neq 2,4$,  the total spaces $X_i,i=1,2$ of two $S^{2k-1}$-fibrations  over $S^{2k}$ with orientations defined in (\ref{Def:Orien of X}) are not orientation reversing homotopy equivalent. 
 \end{theorem}

  For $k=3,5$ and $6$, we  classify the homotopy types of the total spaces of $S^{2k-1}$- fibrations  over $S^{2k}$ by listing the attaching map  of the top cell  when the total space is regarded as a CW complex.
 
 \begin{theorem}\label{Thm: Gkn}
 	$P^{2k}(n)\cup_{f}e^{4k-1}$ is homotopy equivalent to  the total space of any $S^{2k-1}$-fibration over $S^{2k}$ if and only if $f$ is in the following list:
 	\begin{enumerate}[$\bullet$]
 		\item for $k=3$,
 		\newline if $2|n$,
 	   then	$G_3^n=1$ and 	$f=\theta_{3}^n$;

 		\item for $k=5$	
 		\newline if $2|n$, then $G_5^n=8$ and $f=\theta_{5}^n+i_9\bar\xi$, $\bar\xi\in S_0$;
 		\item for $k=6$	 
 		\newline if  $2||n$, then $G_6^n=2(1+\rho_{3}^n+3\rho_{9}^n)(1+3\rho_{7}^n)$ and

 		\qquad $f=\theta_6^n+b_1i_{11}\zeta_{11}+ c \rho_{3}^n i_{11} \alpha_{3}^{11}+e\rho_{7}^n i_{11}\alpha_{1}^{11}$;
 		\newline if $4||n$, then $G_6^n=3(1+\rho_{3}^n+3\rho_{9}^n)(1+3\rho_{7}^n)$ and
 		
 		\qquad  $f=\theta_6^n+b_2i_{11}\zeta_{11}+ c \rho_{3}^n i_{11} \alpha_{3}^{11}+e\rho_{7}^n i_{11} \alpha_{1}^{11}$; 
 		\newline  if $8|n$, then $G_6^n=5(1+\rho_{3}^n+3\rho_{9}^n)(1+3\rho_{7}^n)$ and 
 		
 		\qquad	$f=\theta_6^n+b_4i_{11}\zeta_{11}+ c \rho_{3}^n i_{11} \alpha_{3}^{11}+e\rho_{7}^n i_{11} \alpha_{1}^{11}$.
 	\end{enumerate}
  where $b_i\in \{0,1,\cdots,i\}$,   $c\in \{0,1\}$ or $\{0,1,2,3,4\}$ according as $3||n$ or $9|n$,  $e\in\{ 0,1,2,3\}$.
 \end{theorem}

This paper is organized as follows. Section \ref{sec:compute K_k^n} computes the group $K_k^n$.   In Section \ref{sec:Number Theory}, we give some lemmas in  number theory which will be used in classifying the sphere bundles. Section \ref{sec:Action E(X)} studies the action of the group of self-homotopy equivalences of $P^{2k}(n)$ on attaching map $f$. Sections \ref{sec:S7-bundle over S8} and  \ref{sec:fibrations for k=3,5,6} give proofs of the main results about the homotopy classification.

\section{Computation of $K_k^n$}
\label{sec:compute K_k^n}

In this section,  we calculate $K^n_k=Ker(p_{2k\ast}:\pi_{4k-2}(\Pn)\to \pi_{4k-2}(S^{2k}))$. 

The following Lemma comes from \cite[Theorem 1.7 (i)]{ZhuPanfiber}
\begin{lemma}\label{lem:ord of theta}
	For a finite abelian group $A$, denote $o(A)=min\{ \text{positive integer}~a|~aA=0\}$ as the order of $A$ and $o_{k}^n=o(K_{k}^n)$.
	\begin{align*}
		\mathbf{order} (\theta_{k}^n)=o_{k}^n=\left\{
		\begin{array}{ll}
			n, & \hbox{$2\nmid n$ or $k=2, 8|n$ or $k=4, 16|n$;} \\
			4n, & \hbox{$2\nmid k$, and $2||n$;}	
			\\
			2n, & \hbox{otherwise.}
		\end{array}
		\right.
	\end{align*}
	Moreover if $2|n$, then  $K_{k}^{n}\cong \Z_{o_{k}^n}\oplus A$ with $2 o(A)|o_{k}^n$. 
\end{lemma}

For $k=2,4$,  there is  the following commutative diagram  \cite[ Proposition 5.82(ii)]{Ranicki}
\begin{equation}\label{diam: J Suspension}
	\xymatrix{
		\pi_{2k-1}(SO(2k-1))\ar[d]_{J} \ar[r]^{\mathbf{i}_\ast} &\pi_{2k-1}(SO(2k))\ar[d]_{J}\\
		\pi_{4k-2}(S^{2k-1}) \ar[r]^{\Sigma}  &\pi_{4k-1}(S^{2k}) }
\end{equation}
where $\Sigma$ is the suspension map and   $J$ is the $J$-homomorphism given in G.W.Whitehead \cite{GWWhitehead}. Since  its stablelization  $J:\mathbb{Z} \cong \pi_{2k-1}(SO) \rightarrow \pi^s_{2k-1}$ is surjective  by \cite[Theorem 1.5,1.6 ]{Adams},  the left  $J$-homomorphism  in (\ref{diam: J Suspension}) is also surjective. Thus we get 
\begin{align}
	& \pi_{4k-2}(S^{2k-1})=\Z_{n_{2k-1}}\{\xi_{2k-1}\} ~\text{with}~\xi_{2k-1}=J(\bar \rho_{2k-1}), \label{equ： xi3,xi7}
\end{align}
where $n_{2k-1}=12$ or $120$ according as $k=2$ or $4$. 

 From \cite[Lemma 2.2]{ZhuPanfiber}, 
\begin{align}
	&j_{\ast}(K^n_k)=\Z_{n_2}\{[X_{2k}, \iota_{2k-1}]_r\},  ~\text{where }~ n_2\!=\!\left\{\!\!
	\begin{array}{ll}
		n, &\! \hbox{$k=2,4$;} \\
		2n, &\! \hbox{$k\neq 2,4$ and $2|n$.}
	\end{array}
	\right. \label{equ:i(K)}
\end{align}
Thus,  sequence (\ref{exact: K}) becomes the following two short exact sequences 
\begin{align*}
&0 \to  \Z_{(n_{2k-1},n)}\{i_{2k-1}\xi_{2k-1}\} {\to}  K^n_k \overset{j_\ast}{\to} \Z_{n}\{[X_{2k},\iota_{2k-1}]_r\} \to 0, ~~k=2,4
\end{align*}
where  $\Z_{(a,b)}=\Z_1=0$ for $(a,b)=1$.

 The following results of  $K^n_k$ for $k=2,4$ are obtained by Lemma (4) and Theorem (iii) of \cite{SASAO2} 
\begin{lemma}\label{Lem:htpygp of K2,4}
	\begin{align*}
		&K^n_2=\left\{
		\begin{array}{ll}
	\Z_{(12,n)}\{i_{3}\xi_{3}\}	\oplus 	  \Z_{n}\{\theta_{2}^{n}\}, & \hbox{$2\nmid n$ or $ 8|n$;} \\
		\Z_{\frac{(n,12)}{2}}\{\frac{2n}{(n,12)}\theta_{2}^{n}+i_{3}\xi_{3}\}\oplus \Z_{2n}\{\theta_{2}^{n}\}, & \hbox{$2||n ~\text{and}~8\nmid n$.}
		\end{array}
		\right.\\
		&K^n_4=\left\{
		\begin{array}{ll}
				\Z_{(120,n)}\{i_{7} \xi_7\}\oplus 	\Z_{n}\{\theta_{4}^{n}\} , & \hbox{$2\nmid n$ or $ 16|n$;}\\
		\Z_{\frac{(n,120)}{2}}\{\frac{2n}{(n,120)}\theta_{4}^{n}+i_{7}\xi_{7}\}\oplus \Z_{2n}\{\theta_{4}^{n}\},& \hbox{$2|n$ and $16\nmid n$;} 
		\end{array}
		\right.\\
	 &	n\theta_{k}^{n}=\frac{n(n-1)}{2}i_{2k-1}\xi_{2k-1}, k=2,4. 
	\end{align*}
\end{lemma}

For $ k\neq 2,4$ and $2|n$, by (\ref{equ:i(K)}),  the exact sequence (\ref{exact: K}) becomes 
\begin{align}
	0 \to  \pi_{4k-2}( S^{2k-1})/n\pi_{4k-2}( S^{2k-1}) \overset{i_{2k-1\ast}}{\to}  K^n_k \overset{j_\ast}{\to} j_{\ast}(K^n_k) \to 0   \label{Exact: K}
\end{align}

From (\ref{equ:i(K)})  and  Lemma \ref{lem:ord of theta},  we get 
\begin{proposition}\label{Prop:htpygp of K neq2,4}
	Let $k\neq 2,4$ and $2|n$. 
	\begin{enumerate}[(i)]
		\item \label{o(theat) odd k 2||n} for odd $k$ and  $2||n$, 	$\mathbf{order} (\theta_{k}^n)=4n$ and $2n\theta_{k}^n=i_{2k-1}\xi$ for some $\xi\in \pi_{4k-2}(S^{2k-1})$;
		\item \label{o(theat) otherwise} otherwise,  the exact sequence (\ref{Exact: K}) is split and 
		\begin{align*}
			K^n_k=i_{2k-1\ast}\pi_{4k-2}( S^{2k-1})/n i_{2k-1\ast}\pi_{4k-2}( S^{2k-1})\oplus \Z_{2n}\{\theta_{k}^{n}\}. 
		\end{align*}
	\end{enumerate}
\end{proposition}

\begin{lemma}\label{Lem:K3,5}
	\begin{align*}
&K^n_3=\left\{
		\begin{array}{ll}
			\Z_{4n}\{\theta_{3}^{n}\}~\text{with}~i_5\nu_5\eta_{8}^2=2n\theta_3^n, & \hbox{$2||n$;} \\
		\Z_2\{i_{5}\nu_5\eta_{8}^2\}\oplus 	\Z_{2n}\{\theta_{3}^{n}\}, & \hbox{$ 4|n$.}
		\end{array}
		\right.\\
&K^n_5=\left\{
\begin{array}{ll}
	 \Z_2\{i_9 \nu_9^3\}\oplus \Z_2\{i_9 \mu_9\}\oplus \Z_2\{i_9 \eta_9\varepsilon_{10}\}\oplus \Z_{4n}\{\theta_{5}^{n}\}, & \hbox{$2||n$;} \\
\Z_2\{i_9 \nu_9^3\}\oplus \Z_2\{i_9 \mu_9\}\oplus \Z_2\{i_9 \eta_9\varepsilon_{10}\}\oplus \Z_2\{i_9 \sigma_9\eta_{16}^2\}\oplus \Z_{2n}\{\theta_{5}^{n}\}, & \hbox{$ 4|n$.}
\end{array}
\right.\\
&	~~~~~~~~~~~\text{with}~i_9\nu_9^3+i_9\eta_9\varepsilon_{10}+i_9\sigma_9\eta_{16}^2=2n\theta_{5}^n~~\text{for}~2||n.\\
&K^n_6=\Z_{(8,n)}\{i_{11} \zeta_{11}\}\!\oplus \!\Z_{(9,n)}\{i_{11} \bar\alpha^{11}_{3}\}\!\oplus \! \Z_{(7,n)}\{i_{11} \alpha^{11}_{1}\}\!\oplus \!	\Z_{2n}\{\theta_{6}^{n}\},  2|n.
	\end{align*}
\end{lemma}
	\begin{proof}
For	$K^n_3$,  $K^n_5$ with $4|n$ and $K^n_6$  with $2|n$ are obtained  from  Proposition  \ref{Prop:htpygp of K neq2,4} (\ref{o(theat) otherwise}).
		
For $2||n$ and  $k=3$,  from (\ref{Exact: K}) and Proposition \ref {Prop:htpygp of K neq2,4} (\ref{o(theat) odd k 2||n}), we get     $K_3^n=\Z_{4n}\{\theta_{3}^{n}\}$ with $i_5\nu_5\eta_{8}^2=2n\theta_3^n$. 

Next, we determine the $K^n_5$ for $2||n$. Firstly, for $P^{10}(n)$ with $2|n$, 
there is the following commutative diagram, where $K_k^{n,s}:=Ker(\pi_{4k-2}^s(P^{2k}(n))\xrightarrow{p_\ast}\pi_{4k-2}^s(S^{2k}))$
\begin{align}
	\small{\xymatrix{
\Z_{2}\{\nu_9^3\} \!\oplus \!\Z_{2}\{\mu_9\}\!\oplus \! \Z_{2}\{\eta_9\varepsilon_{10}\}\ar[rd]^{\Sigma^{\infty}\cong}\ar@{^{(}->}[r] & \pi_{18}(S^9) \ar[d]^{\Sigma^{\infty}}\ar[r]^{i_{9\ast}} & K^n_5\ar[d]^{\Sigma^{\infty}} \\
\!\!\!\!\Z_{2}\{\nu^3\}\!\oplus \! \Z_{2}\{\mu\}\!\oplus \! \Z_{2}\{\eta\varepsilon\}\ar@{=}[r] 	& \pi_{18}^s(S^9) \ar@{^{(}->}[r]^-{i^s_{9\ast}}_-{\cong}& K_5^{n,s}
} }	\label{Diam: K,Ks}
\end{align}
where $\Sigma^{\infty}$ denotes stabilization (from a homotopy group, or subgroup, to its stable counterpart), and $f^s=\Sigma^{\infty}f$ for a map $f$. Since the left map $\Sigma^{\infty}$ and $i^s_{9\ast}$ in the above  diagram are isomorphisms,  the inclusion $\Z_{2}\{i_9\nu_9^3\} \oplus \Z_{2}\{i_9\mu_9\}\oplus \Z_{2}\{i_9\eta_9\varepsilon_{10}\}\hookrightarrow K^n_5$ has a left inverse. This implies it is a direct summand of $K^n_5$. 

From  exact sequence (\ref{Exact: K}) and  Proposition \ref {Prop:htpygp of K neq2,4} (\ref{o(theat) odd k 2||n}), 
\begin{align*}
	K_5^n=	 \Z_2\{i_9 \nu_9^3\}\oplus \Z_2\{i_9 \mu_9\}\oplus \Z_2\{i_9 \eta_9\varepsilon_{10}\}\oplus\Z_{4n}\{\theta_{5}^{n}\}~\text{for}~2||n, 
\end{align*}
 with $i_{9\ast}(\xi_0)=2n\theta_{5}^{n}$ for some $\xi_0\in \pi_{18}(S^9)$.   By the commutative square in diagram (\ref{Diam: K,Ks}), $i_{9\ast}^s(\Sigma ^{\infty}(\xi_0))=\Sigma^{\infty}(i_{9\ast}(\xi_0))=2n\Sigma^{\infty}(\theta_{5}^{n})=0$, implying 
\begin{align}
	\Sigma^{\infty}(\xi_0)=0. \label{equ:stable xi 0}
\end{align}
By $\Sigma^{\infty}(\sigma_9\eta_{16}^2)=\nu^{3}+\eta\varepsilon$ (\cite[Lemma 6.3, 6.4, Equation (7.5), Theorem 7.2]{Toda}), $\xi_0=\nu_9^3+\eta_9\varepsilon_{10}+\sigma_9\eta_{16}^2$.  Hence 
$i_9\nu_9^3+i_9\eta_9\varepsilon_{10}+i_9\sigma_9\eta_{16}^2=2n\theta_{5}^{n}$.
	
	\end{proof}

\section{Some results in Number Theory}
\label{sec:Number Theory}
The following  lemmas and propositions in Number Theory  are useful in the proof of  Theorem \ref{Thm: Gkn}

 \begin{lemma}\label{Lem: Number theory 1}
	For $r\geq 3$, $x^2\equiv 1$ (mod $2^r$) if and only if $x\equiv \pm 1, \pm 5^{2^{r-3}}$ (mod $2^r$). 
\end{lemma}
\begin{proof}
	By Theorem 4.6 of \cite{Numbertheory}, it is well known that
	\begin{align*}
		\Z_{2^r}^{\ast}=\{\pm 5^{w}~|0\leq w<2^{r-2}\}. 	
	\end{align*} 
	So $x=\pm 5^w$ for some $0\leq w<2^{r-2}$.
	$x^2\equiv 1$ (mod $2^r$) if and only if $5^{2w}\equiv 1$ (mod $2^r$). That is 
	$2w=2^{r-2}u, u\geq 0$.  Hence  $u=0,1$, i.e., $w=0$ or $2^{r-3}$. 
\end{proof}

\begin{lemma}\label{Lem: Number theory x2=pm 1} 
Let $n=2^rp_1^{e_1}p_2^{e_2}\cdots p_s^{e_s}$ be the prime factorization of $n$.
	\begin{enumerate}[(i)]
		\item \label{NumberTheo x^2=1}~$x^2\equiv 1$ (mod $p^{e}$) for odd prime $p$ has two solutions  ~$x\equiv \pm 1$ (mod $p^{e}$).
		\item \label{NumberTheo x^2=1 solution}~Let $r'\leq min\{3,r-1\}$, $e'_{i}\leq e_i$ ($i=1,2,\cdots,s$). 
		
		If $t^2\equiv 1$ (mod $n$), then $t\equiv \pm 1$ (mod $2^{r'}$), $t\equiv \pm 1$ (mod $p_i^{e'}$), $i=1,2,\cdots,s$;
		
		Conversely, for any set of integers $\epsilon_i\in \{1,-1\}$,  $i=0,1,\cdots, s$, there exists $t$ with $t^2\equiv 1$ (mod $n$) such that  $t\equiv \epsilon_0$ (mod $2^{r'}$), $t\equiv \epsilon_i$ (mod $p_i^{e'}$), $i=1,2,\cdots,s$. 
		
		\item \label{NumberTheo x^2=-1}~ The congruence $x^2\equiv -1$ (mod $n$) is solvable if and only if  $r=0,1$ and $p_i\equiv 1$ (mod $4$), $i=1,\cdots, s$.
		\item \label{NumberTheo x^2=-1 5|n}~Let $5|n$ and  $x^2\equiv -1$ (mod $n$) be solvable.  Then any  solution $\alpha$ satisfies  one of $\alpha\equiv 2$  (mod $5$) or  $\alpha\equiv -2$  (mod $5$); Conversely, for any $\epsilon\in \{2,-2\}$,  there is solution $\alpha$ of  $x^2\equiv -1$ (mod $n$) such that $\alpha\equiv \epsilon$ (mod $5$). 

        	\end{enumerate}
\end{lemma}
\begin{proof}
	(\ref{NumberTheo x^2=1}) comes from the Theorem 5.2 of \cite{Numbertheory};
	
    For (\ref{NumberTheo x^2=1 solution}), firstly 
	by Theorem 3.21 of  \cite{Numbertheory},  $x^2\equiv 1$ (mod $n$) is equivalent to the simultaneous system  $x^2\equiv 1$  (mod $2^r$),  $x^2\equiv 1$  (mod $p_1^{e_1}$), $\cdots$,  $x^2\equiv 1$  (mod $p_s^{e_s}$).  Thus $t\equiv \pm 1$ (mod $2^{r'}$), $t\equiv \pm 1$ (mod $p_i^{e'}$), $i=1,2,\cdots,s$ are obtained  by Lemma \ref{Lem: Number theory 1} and  Lemma \ref{Lem: Number theory x2=pm 1} (\ref{NumberTheo x^2=1}).
    For the ``Conversely" part,  consider the system 	$x\equiv \epsilon_0$ (mod $2^r$), $x\equiv  \epsilon_i$ (mod $p_i^{e_i}$), $i=1,\cdots, s$. By the Chinese Remainder Theorem, this system  has unique solution $t\in \Z_n$ such that   $t^2\equiv 1$  (mod $n$). Clearly, $t$ satisfies the congruence equations.

	(\ref{NumberTheo x^2=-1}) is easily obtained by Theorem 3.21 and 3.22 of \cite{Numbertheory}. 
	
 For (\ref{NumberTheo x^2=-1 5|n}),  the first part holds since  $x^2\equiv -1$ (mod $n$) $\Rightarrow$  $x^2\equiv -1$ (mod $5$)  $\Rightarrow$  $x\equiv \pm 2$ (mod $5$).  For the ``Conversely" part,  Assume that  $5^e ||n$. Since 	$\epsilon^2 \equiv -1$ (mod $5$), by Hensel's Lemma \cite[Theorem 3.19]{Nathanson Number Theory},  there is an integer $\alpha_5$ such that 	$\alpha_5^2 \equiv -1$ (mod $5^e$) and $\alpha_5\equiv \epsilon$ (mod $5$). Then the existence of $\alpha$ is obtained by the Chinese Remainder Theorem. 
\end{proof}

\begin{lemma}\label{Lem: Number theory x^2=1+n 8|n} Let $2|n$. Then
	 
		\begin{enumerate}[(i)]
		\item \label{x^2=1+n 8 nmid n}
	The congruence $x^2\equiv 1+n$ (mod $2n$) has no solution for $8\nmid n$.
		\item \label{x^2=1+n 8 ||n}
		  $x^2\equiv 1+n$ (mod $2n$) is solvable for  $8||n$. Moreover, for $n=8p_1^{e_1}p_2^{e_2}\cdots p_s^{e_s}$ and any solution $\alpha$ of $x^2\equiv 1+n$ (mod $2n$), we have $\alpha\equiv \pm 5$ (mod $8$) and $\alpha\equiv \pm 1$ (mod $p_i$) for $i=1,\cdots, s$. 
		\end{enumerate}
\end{lemma}
\begin{proof}
	 (\ref{x^2=1+n 8 nmid n}) is easy to be obtained by the fact that $8|a^2-1$ for odd integer $a$. 
	
	For (\ref{x^2=1+n 8 ||n}), let $n=8p_1^{e_1}p_2^{e_2}\cdots p_s^{e_s}$. Consider the system $x\equiv 5$ (or $-5$) (mod $8$); $x\equiv 1$ (or $-1$)(mod $p_i^{e_i}$), $i=1,\cdots,s$.  By the Chinese Remainder Theorem, this system  has a unique solution $\alpha\in \Z_n$, satisfying    $\alpha^2\equiv 1$  (mod $n$).  Since   $\alpha=5+8l$ or $-5+8l$ for some integer $l$, $\alpha^2-1=8l_0$ for some odd integer $l_0$. This implies that  $\alpha^2-1=u_0n$ for some odd  integer $u_0$. Thus $\alpha^2\equiv 1+u_0n\equiv 1+n$ (mod $2n$). The proof of ``Moreover" part is just the same as that of Lemma \ref{Lem: Number theory x2=pm 1}(\ref{NumberTheo x^2=1 solution}).
\end{proof}

\begin{lemma}\label{Lem: same equ for m m'} 
If $8\nmid d$, $d |n$ and $u\in \Z$  such that $u^2 \equiv 1 ~(\text{mod}~ d)$, then there is a  $t\in \Z$ such that $t^2 \equiv 1 ~(\text{mod}~ n)$ and $t \equiv u ~(\text{mod}~ d)$. 
\end{lemma}
\begin{proof}
	Let $d=2^{r}p_1^{e_1}\cdots p_s^{e_s}$  denote the prime factorization of $d$.  By $d|n$, we get $n=2^{r'}p_1^{e'_1}p_2^{e'_2}\cdots p_s^{e'_s}n_0$ where $(n_0, 2p_1p_2\cdots p_s)=1$ and 
	$r\leq \min\{2, r'\}$,  $e_i\leq e_i'$, $i=1,2,\cdots,s$. 
	
	Without loss of generality, we assume $r\geq 1$; the proof for $r=0$ is similar and easier.  The congruence equation  $x^2 \equiv 1 ~(\text{mod}~ d)$ is equivalent to the following system of congruences
	\begin{align}
		\left\{
		\begin{array}{ll}
			x^2 \equiv 1 ~(\text{mod}~ 2^r),~\\
			x^2 \equiv 1 ~(\text{mod}~ p_{i}^{e_{i}}), i=1,\cdots,s.~
		\end{array}
		\right.  ~\Leftrightarrow  \left\{
		\begin{array}{ll}
			x \equiv \pm 1 ~(\text{mod}~ 2^r),~\\
			x\equiv \pm 1 ~(\text{mod}~ p_{i}^{e_{i}}), i=1,\cdots,s.~
		\end{array}
		\right.    \label{System for u}
	\end{align}
	where the equivalence ``$\Leftrightarrow$"  above is obtained by Theorem 5.2 of \cite{Numbertheory} for $r\leq  2$. 
	
	We  assume that  
		\begin{align*}
		u \equiv (-1)^{u_0} ~(\text{mod}~ 2^{r}), ~~	u \equiv (-1)^{u_i}~(\text{mod}~ p_{i}^{e_{i}}), i=1,\cdots,s. 
	\end{align*}
	Consider the following  system of congruences
	\begin{align*}
		t \equiv (-1)^{u_0} ~(\text{mod}~ 2^{r'}), 	t \equiv (-1)^{u_i}~(\text{mod}~ p_{i}^{e'_{i}}), i=1,\cdots,s, ~\text{and}~ t \equiv 1~(\text{mod}~ n_0). 
	\end{align*}
Above system has a unique solution $t\in \Z_n$.  Obviously, it satisfies  $t^2 \equiv 1 ~(\text{mod}~ n)$ and  $t \equiv u ~(\text{mod}~ d)$. 
\end{proof}
\begin{remark}
	For $8|d$, Lemma \ref{Lem: same equ for m m'} dose not hold  since  in (\ref{System for u}), it is known from Lemma \ref{Lem: Number theory 1} that we can not get $x \equiv \pm 1 ~(\text{mod}~ 2^r)$ by $x^2 \equiv  1 ~(\text{mod}~ 2^r)$  for $r\geq 3$.    
\end{remark}

\section{The action of the homotopy equivalences} 
\label{sec:Action E(X)}

Let $X=\Pn\cup_{f}e^{4k-1}$  be the total space of an $S^{2k-1}$-sphere fibration over $S^{2k}$ and the attaching map $f$ come from the following set 
\begin{align*}
	I^n_{k}:=\{f=\theta_{k}^{n}+i_{2k-1}\xi~|~\xi\in\pi_{4k-2}(S^{2k-1})\}.
\end{align*}
 Let $Aut(X)$ be the group of self-homotopy equivalences of $X$.

\begin{proposition}\label{Prop: orient hty equi of X}
Let $X_{t}=P^{2k}(n)\cup_{f_t} e^{4k-1},{f_t}=\theta^n_k +i_{2k-1}\xi_t$ be CW complexes with $\iota$-compatible orientation classes $\sigma_{f_t},t=1,2$ where $\iota$ is a generator of $H_{4k-1}(S^{4k-1})$. Then $X_{1}$ and $X_{2}$ are orientation preserving homotopy equivalent iff
there exists a  $g\in Aut(P^{2k}(n))$ such that $g f_1 = f_2$.
\end{proposition}
\begin{proof}
	If $h:X_{1} \to X_{2}$ is an orientation preserving homotopy equivalence, then $h$ can be assumed to be a cellular map and $h$ restricts to a homotopy equivalence $g:P^{2k}(n)\to P^{2k}(n)$.
	 There is a map $\bar{g}:S^{4k-2}\to S^{4k-2}$ such that the following diagrams are homotopy commutative

	$$	\small{\xymatrix{
			S^{4k-2}\ar[r]^{f_1}\ar[d]^{ \bar{g}} 	&P^{2k}(n)\ar[d]^-{g} \ar[r]^-{i_1}&X_1\ar[r]^-{q_1}\ar[d]^{h}&S^{4k-1}\ar[d]^{\Sigma \bar{g}} \\
			S^{4k-2}\ar[r]^{f_2}&	P^{2k}(n) \ar[r]^-{i_2}& X_2 \ar[r]^{q_2} & S^{4k-1}
	} }. $$
	By definiton, $ (q_t)_\ast (\sigma_{f_t})=\iota$, $t=1,2$, where $\iota$ is a fixed generator of $H_{4k-1}(S_{4k-1})$. 
	
		It follows that $(\Sigma \bar{g})_\ast (q_1)_\ast (\sigma_{f_1})=(q_2)_\ast h_\ast (\sigma_{f_1})=(q_2)_\ast (\sigma_{f_2})=\iota$. 
Hence, $(\Sigma \bar{g})_\ast (\iota)=\iota$ and  $\Sigma \bar{g} = \iota_{4k-1}$.  These imply  $\bar{g} =\iota_{4k-2}$ and 	$gf_1 = f_2$.

	On the other hand if there exists a homotopy equivalence $g:P^{2k}(n)\to P^{2k}(n)$ such that $g f_1 = f_2$,
	then $g$ induces a homotopy equivalence $h:X_{1} \to X_{2}$
	such that the following diagram is homotopy commutative
	
	$$
	\small{\xymatrix{
			S^{4k-2}\ar[r]^{f_1}\ar[d]^{\iota_{4k-2}} 	&P^{2k}(n)\ar[d]^-{g} \ar[r]^-{i_1}&X_1\ar[r]^-{q_1}\ar[d]^{h}&S^{4k-1}\ar[d]^{\iota_{4k-1}} \\
			S^{4k-2}\ar[r]^{f_2}&	P^{2k}(n) \ar[r]^-{i_2}& X_2 \ar[r]^{q_2} & S^{4k-1}
	} }. 
	$$
	
	A simple diagram chasing shows that $h_\ast (\sigma_{f_1})=\sigma_{f_2}$, so $h$ is an orientation preserving homotopy equivalence.
	
\end{proof}

\begin{remark}\label{rem:orientation}~~
	
	\begin{enumerate}[(i)]
		\item It is clear that in Proposition \ref{Prop: orient hty equi of X},   $X_{1}$ and $X_{2}$ are orientation reversing homotopy equivalent if and only if 
		there exists a  $g\in Aut(P^{2k}(n))$ such that $g f_1 = -f_2$.
		\item The orientation preserving homotopy classification of the total spaces of $S^{2k-1}$-fibrations over $S^{2k}$ reduces to that of oriented CW complexes  $(X_f,\sigma_f)$ with $X_f=P^{2k}(n)\cup_{f} e^{4k-1},f\in I_k^n$.
	\end{enumerate}

\end{remark}
For  $f,f'\in \pi_{4k-2}(\Pn)$,  denote $f\sim f'$ if $	\Pn\cup_{f}e^{4k-1}\simeq \Pn\cup_{f'}e^{4k-1}$. Otherwise,  denote $f\nsim f'$. Clearly, $f\sim -f$ and  ``$\sim$'' defines an equivalence relation  on the set $I_k^n$.  Let $I_{k/\sim}^n$ denote the set of the equivalence classes, and  $[f]$ the class containing $f$.

 Let  $\iota_P$ (simplifying $\iota^{k+1}_P$) denote the identity map on Moore space $P^{k+1}(n)$.  By Corollary 1.4.10 of \cite{Baues} and Lemma (5) of \cite{SASAO2}, 
$$[P^{k+1}(n), P^{k+1}(n)]=\left\{
\begin{array}{ll}
	\Z_n\{\iota_P\}, & \hbox{$n$  odd;} \\
	\Z_{2n}\{\iota_P\},~\text{with}~i_k\eta_k p_{k+1}=n\iota_P,& \hbox{$ 2||n$;}
	\\
	\Z_n\{\iota_P\}\oplus \Z_2\{i_k\eta_k p_{k+1} \}, & \hbox{$4|n$.}	
\end{array}
\right.$$
From Lemma (7) of \cite{SASAO2}, 
$$Aut(P^{k+1}(n))=\left\{
\begin{array}{ll}
	\{t\iota_P | t\in \Z^{\ast}_n\}, & \hbox{$n$  odd;} \\
		\{t\iota_P | t\in \Z^{\ast}_{2n}\},~\text{with}~i_k\eta_k p_{k+1}=n\iota_P, & \hbox{$ 2||n$;}
	\\
	\{ t\iota_P+\epsilon i_k\eta_k p_{k+1}~|~t\in \Z_{n}^{\ast},\epsilon\in \{0,1\}\}, & \hbox{$4|n$.}	
\end{array}
\right.$$

	For $k=2,4$, \cite[Theorem (4)]{SASAO2} provides  $\theta_k^{n}\in K_k^n$  with  $j_{\ast}(\theta_k^{n})=[X_{2k},\iota_{2k-1}]_r$  satisfying  
	\begin{align}
		(2\iota_P) \theta_k^{n}=4\theta_k^{n}~ (k=2,4) \label{equ: 2theta_k^n k=2,4}.
	\end{align}

 \begin{lemma}\label{Lem: composition any k}
 	
  For $P^{2k}(n)$,	 $ t$ an integer such that $(t,n)=1$,  $\xi\in \pi_{4k-2}(S^{2k-1})$,the followings are true:
 
 	\begin{enumerate}[(i)]
 		\item \label{comp any k}~$(t\iota_P) (\theta_k^{n}+i_{2k-1}\xi)=t^2\theta_k^{n}+i_{2k-1}\xi''$ for some $\xi''\in \pi_{4k-2}(S^{2k-1})$.
 		\item \label{comp k=2,4}~If  $k=2,4$ and $\theta_k^{n}$ satisfies  (\ref{equ: 2theta_k^n k=2,4}), then 
 	\begin{align*}
 		(t\iota_P+\epsilon i_{2k-1}\eta_{2k-1} p_{2k}) (\theta_k^{n}+i_{2k-1}\xi)=t^2\theta_k^{n}+ti_{2k-1}\xi,~~ \xi\in\pi_{4k-2}(S^{2k-1}). 
 	\end{align*}
 	\end{enumerate}
\end{lemma}
\begin{proof}
	Lemma \ref{Lem: composition any k}  (\ref{comp any k}) is from the following equations 
	\begin{align*}
		j_{\ast}((t\iota_P) (\theta_k^{n}+i_{2k-1}\xi))=(t\iota_P)_{\ast}j_{\ast}(\theta_k^{n}+i_{2k-1}\xi)=(t\iota_P)_{\ast}([X_{2k},\iota_{2k-1}]_r)=t^2[X_{2k},\iota_{2k-1}]_r.  
	\end{align*}
Substituting condition  (\ref{equ: 2theta_k^n k=2,4})  into the proof of Theorem (4) in \cite{SASAO2}  readily yields Lemma \ref{Lem: composition any k}  (\ref{comp k=2,4}). 
\end{proof}

 \section{Classification of $S^{2k-1}$-bundles over $S^{2k}$ by characteristic maps for $k=2,4$}
 \label{sec:S7-bundle over S8}
 In this section, we always assume $k=2$, $4$.  By the definition,  
 \begin{align}
 	&\mathbf{i}_{\ast}(\bar \rho_{2k-1})=\bar \rho_{2k}, ~~~~~~~~~	\mathbf{p}_{\ast}(\bar\sigma_{2k})=\iota_{2k-1}. \nonumber\\
 	\text{Hence,}~~~~&	\mathbf{i}_{\ast}(m\bar \rho_{2k-1})=[m,0] ~~\text{and}~~	\mathbf{p}_{\ast}([m,n])=n\iota_{2k-1}.  \nonumber\\
 	     \text{Clearly,}~~~~    &  	[m,n]-[m',n]=(m-m')\mathbf{i}_{\ast}(\bar \rho_{2k-1}).  \label{equ:(m,n)-(m',n)}
 \end{align}
 From \cite[Section 3]{JamesI},  $M^{4k-1}_{m,n}$ has a cell decomposition
\begin{align}
	M^{4k-1}_{m,n}=P^{2k}(n)\cup_{f_{m,n}} e^{4k-1}  \label{Cell decom of M}
\end{align}  
 where the attaching map  $f_{m,n}$  of the top cell of $M^{4k-1}_{m,n}$ is defined in \cite[Section 3]{JamesI}. Then 
 we have the following
 
 \begin{proposition}\label{pro:fm,n}~ 
 	
 	\begin{enumerate}[(i)]
 	\item \label{condition fmn:i} Given any generator  $X_{2k}\in \pi_{2k}(\Pn, S^{2k-1})$,   there is a $g^n_{m,m'}\in Aut(P^{2k}(n))$  such that 
 	\begin{align*}
 			(g^n_{m,m'})_{\ast}([X_{2k},\iota_{2k-1}]_r)=[X_{2k},\iota_{2k-1}]_r~~\text{and}~~g^n_{m,m'} f_{m,n}-f_{m',n}=  (m-m')i_{2k-1} \xi_{2k-1}.
 	\end{align*}
 	\item \label{condition fmn:ii} There is a generator $X_{2k}\in \pi_{2k}(\Pn, S^{2k-1})$ such that  $j_{\ast}(f_{m,n})=[X_{2k},\iota_{2k-1}]_r$   for any integer $m$. 
 \end{enumerate}
 \end{proposition}

 \begin{proof}
 From the \cite[page 156]{JamesII},  there is  a $g^n_{m,m'}\in Aut(P^{2k}(n))$ such that the restriction to $S^{2k-1}$ is the identity map and $(g^n_{m,m'})_{\ast}(X_{2k})=X_{2k}$. The first equation in (\ref{condition fmn:i})  holds by  the naturality of the relative Whitehead product, and  the second follows from \cite[(3.3)]{JamesII}.  For  (\ref{condition fmn:ii}),  there is a generator  $X_{2k}\in \pi_{2k}(\Pn, S^{2k-1})$ such that $j_{\ast}(f_{0,n})=[X_{2k},\iota_{2k-1}]_r$  by \cite[(5.1)]{JamesII}. For any integer $m\neq 0$,  consider $g_{0,m}$ in (\ref{condition fmn:i}) of Proposition \ref{pro:fm,n}, we  also have  $j_{\ast}(f_{m,n})=[X_{2k},\iota_{2k-1}]_r.$ 
 \end{proof}

 \begin{lemma}\label{Lem:orien of M}
 	Let $X_{2k}$ be the  generator in (\ref{condition fmn:ii} ) of Proposition \ref{pro:fm,n}.  Then there exists a generator $\tilde \iota_{4k-1}\in H_{4k-1}(S^{4k-1})$ such that  the orientation class $\sigma^{4k-1}_{m,n}$
 	 of $M_{m,n}^{4k-1}$  defined by the orientations of base space $S^{2k}$ and fibre $S^{2k-1}$,  is $\tilde{\iota}_{4k-1}$-compatible for any $m,n$. 
 \end{lemma}  
 \begin{proof}
 	Let $a^{t}$ be the base point of $S^t$ and $u_t:D^t\rightarrow S^t$  the pinch map that collapses the boundary of the disk to the basepoint. 
 	
 	There is a map $h:D^{2k}\times S^{2k-1}\rightarrow M_{m,n}^{4k-1}$  defined in \cite[(3.2)]{JamesI} with the following  properties
 		\begin{enumerate}[(i)]
 			\item \label{topcell i}  $u_{2k}=\pi (h|_{D^{2k}\times \{y\}}):  D^{2k}\hookrightarrow D^{2k}\times S^{2k-1}\xrightarrow{h} M_{m,n}^{4k-1}\xrightarrow{\pi}  S^{2k}$ for any $y\in S^{2k-1} $;
 			\item \label{topcell ii} $h|_{x}:\{x\}\times S^{2k-1}\rightarrow \pi^{-1}(u_{2k}(x))$  for any $ x\in D^{2k}$ is a homeomorphism  and the orientation of the each fibre  $\pi^{-1}(u_{2k}(x))$ of $M_{m,n}^{4k-1}$ is $(h|_{x})_{\ast}(\iota_{2k-1})$; 
 			\item \label{topcell iii} 
 			$h|_{res}: D^{2k}\times S^{2k-1}\setminus S^{2k-1}\times S^{2k-1} \rightarrow  M_{m,n}^{4k-1}\setminus  \pi^{-1}(a^{2k})$ is a homeomorphism.
 		\end{enumerate}
 		Thus there is a pushout diagram 
 		\begin{align*}
 				\xymatrix{
 				 S^{2k-1}\times S^{2k-1}  \ar@{^{(}->}[d]\ar[r]^-{Pr_{2}} &S^{2k-1} \ar@{^{(}->}[d]^-{i}\\
 			D^{2k}\times S^{2k-1}  \ar[r]^-{h}  & M_{m,n}^{4k-1}}
 		\end{align*}
 		where $Pr_2$ is the projection to the second factor and the left map $i$   is the inclusion map taking $S^{2k-1}$  into $\pi^{-1}(a^{2k})\subset M_{m,n}^{4k-1}$. 
 		
 		So we have a  homeomorphism 
 		\begin{align*}
 			\tilde{h}:	(D^{2k}\times S^{2k-1})\cup_{\phi}S^{2k-1}\rightarrow M_{m,n}^{4k-1}, ~~\phi=h|_{S^{2k-1}\times S^{2k-1}}
 		\end{align*}
 		where $	(D^{2k}\times S^{2k-1})\cup_{\phi}S^{2k-1}=((D^{2k}\times S^{2k-1})\coprod S^{2k-1})/\sim$,  $ (x,y)\sim y'$ if and only if $x\in S^{2k-1}$ and $\phi(x,y)=i(y')$.  
 	
 	The orientations chosen on   $S^{2k-1}$ and  $S^{2k}$	determine the orientations on $D^{2k}\times S^{2k-1}$,  $(D^{2k}\times S^{2k-1})\cup_{\phi}S^{2k-1}$ and   $M_{m,n}^{4k-1}$ (denoted by $\sigma^{4k-1}_{m,n}$) such that the  maps in the following  commutative triangle are orientation preserving
 	\begin{align*}
 	\xymatrix{
 		D^{2k}\times S^{2k-1}  \ar@{^{(}->}[d]\ar[rd]^-{h}& \\
 		(D^{2k}\times S^{2k-1})\cup_{\phi}S^{2k-1} \ar[r]^-{\tilde{h}}  & M_{m,n}^{4k-1}} 
 \end{align*}
From \cite[(3.3)]{JamesI}, in the cell decompsition (\ref{Cell decom of M}), 
 \begin{align*}
 	& e^{4k-1}= h( D^{2k}\times S^{2k-1}\setminus ( D^{2k}\times\{a^{2k-1}\}\cup S^{2k-1}\times S^{2k-1})),\\
 	& P^{2k}(n)= h( D^{2k}\times\{a^{2k-1}\}\cup S^{2k-1}\times S^{2k-1}).
 \end{align*}
 This implies that $\tilde{h}$ induces the following commutative diagram 
 	\begin{align}
 		~~~~~~~	\xymatrix{
 			(D^{2k}\times S^{2k-1})\cup_{\phi}S^{2k-1}  \ar[d]^-{q'} \ar[r]^-{\tilde{h}}& M_{m,n}^{4k-1}\ar[ld]^-{q} \\
 			S^{4k-1}=(D^{2k}\times S^{2k-1})/(D^{2k}\times \{a^{2k-1}\} \cup S^{2k-1}\times S^{2k-1} ) & } \label{diam: orient triangle}
 	\end{align}	
 	where $q$ is the natural quotient map that collapses the $2k$-skeleton of  $M_{m,n}^{4k-1}$ to a point.
 	
 	Thus there is an orientation $\tilde \iota_{4k-1}\in H_{4k-1}(S^{4k-1})$ of $S^{4k-1}$, which is independent of $m$ and $n$,  such that $q'$ is orientation preserving.  From commutative triangle (\ref{diam: orient triangle}), we get $q_{\ast}(\sigma^{4k-1}_{m,n})=\tilde \iota_{4k-1}$, i.e.,  $\sigma^{4k-1}_{m,n}$ is  $\tilde{\iota}_{4k-1}$-compatible. 
 
 \end{proof}

 \begin{lemma} \label{lem:Mmn =M-mn}
  For any $m,n\in \Z$, $M^{4k-1}_{m,n}$ is  orientation preserving diffeomorphic to $M^{4k-1}_{-m-n,n}$, $k=2,4$.
\end{lemma}
\begin{proof}
For $k=4$, this lemma follows from \cite[Corollary 2.18]{S8bundles};  the same proof, with the dimensions of spheres and disks changed accordingly,    also applies to the case $k=2$. 
\end{proof}

\begin{lemma}\label{lemma n=2, k=2,4}
If $n=2$,  there is only one homotopy type $M_{m,2}^{4k-1}$ for $k=2,4$.  
\end{lemma}
\begin{proof}
If $n=2$,  by Lemma \ref{Lem:htpygp of K2,4},    $2\theta_k^2=i_{2k-1}\xi_{2k-1}$ and  $M_{m,2}^{4k-1}\simeq P^{2k}(2)\cup_f e^{4k-1}$ with $f\in \{\theta_k^{2}, \theta_k^{2}+i_{2k-1}\xi_{2k-1} \}$ where $\theta_{k}^2\in K_k^2$ is from  
(\ref{equ: 2theta_k^n k=2,4}). By
 $(-\iota_P)(\theta_k^{2}+i_{2k-1}\xi_{2k-1})=-\theta_k^{2}+i_{2k-1}\xi_{2k-1}=\theta_k^{2}$, we get $\theta_k^{2}+i_{2k-1}\xi_{2k-1}\sim \theta_k^{2}$.  Thus for all $m$, the homotopy type of  $M_{m,2}^{4k-1}$ is the same.
\end{proof}

Let  $\theta_{k}^n\in K_k^n$ be given in (\ref{equ: 2theta_k^n k=2,4}).  Let manifolds $M=P^{2k}(n)\cup_f e^{4k-1}$ with  $f=\theta_k^{n}+mi_{2k-1} \xi_{2k-1}$ and $M'=P^{2k}(n)\cup_{f'} e^{4k-1}$ with $f'=\theta_{k}^{n}+m'i_{2k-1} \xi_{2k-1}$.

\begin{lemma}\label{lem: equi relation}
  Let $k=2,4$.	If $n\neq 2$, then
	\begin{enumerate}[(I)]
		\item \label{lem:t orien}  $M$ and $M'$ are  orientation preserving homotopy equivalent  if and only if  there is an integer $t$ such that $m'\equiv tm+\frac{n(n-1)}{2}\varepsilon  ~\text{(mod $(n_{2k-1}, n)$)}$  and 
		\begin{enumerate}[(1)]	
			\item \label{(I)k=2}  for $k=2$, 	$\left\{
				\begin{array}{ll}
					t^2\equiv 1 ~\text{(mod $2n$)}, &\! \hbox{if $2||n$ or $4||n$;} \\
					t^2\equiv 1 ~\text{(mod $n$)}, &\! \hbox{if  $2\nmid n$ or $8|n$,} 
				\end{array}
				\right.  $
			\item \label{(I)k=4}  for $k=4$, 	$\left\{
			\begin{array}{ll}
				t^2\equiv 1 ~\text{(mod $2n$)}, &\! \hbox{if $2||n$ or $4||n$;} \\
				t^2\equiv 1 ~\text{(mod $n$)}, &\! \hbox{ if  $2\nmid n$ or $16|n$;} \\
				t^2\equiv 1 ~\text{(mod $2n$)~or}~, t^2\equiv 1+n ~\text{(mod $2n$)}, &\! \hbox{if  $8||n$,}
			\end{array}
			\right.  $
			\end{enumerate}
			where \(\epsilon=1\) in the case \(k=4\) and \(8||n\) with  $t^2\equiv 1+n ~\text{(mod $2n$)}$; otherwise \(\epsilon=0\).
			
			\item \label{lem:t unorien}  $M$ and $M'$ are orientation reversing homotopy equivalent  if and only if  $n$ satisfies $\bigstar$ and  there is an integer $t$ such that  $-m'\equiv tm+\frac{n(n-1)}{2}\varepsilon ~\text{(mod $(n_{2k-1}, n)$)}$   and 
			\newline
	  	$\left\{
			\begin{array}{ll}
			\varepsilon=0, 	t^2\equiv -1 ~\text{(mod $2n$)}, &\! \hbox{if $2\nmid n$;} \\
				\varepsilon=1,	t^2\equiv -1+n ~\text{(mod $2n$)}, &\! \hbox{if $2||n$.} 
			\end{array}
			\right.$  
			\newline
	\end{enumerate}
\end{lemma}
\begin{proof}
From  Lemma \ref{Lem: composition any k} (\ref{comp k=2,4}), $M\simeq M'$ if and only if there exists $t\iota_P+\epsilon i_{2k-1} \eta_{2k-1} p_{2k}$ with $(t,n)=1$ (take $\epsilon=0$ when $n$ is odd or $2||n$), such that 
	\begin{align}
		(t\iota_P\!+\!\epsilon i_{2k\!-\!1}\eta_{2k\!-\!1} p_{2k}) (\theta_k^{n}\!+\!mi_{2k-1}\xi_{2k\!-\!1})\!=\!t^2\theta_k^{n}\!+\!tmi_{2k\!-\!1}\xi_{2k\!-\!1}\!=\!\pm ( \theta_k^{n}\!+\!m'i_{2k\!-\!1}\xi_{2k\!-\!1}) \label{equ: tlP compose k=2,4}
	\end{align}
	where the sign ``$+$" and  $``-"$ on the right-hand side of the above equations  represent that the homotopy equivalence on total space is orientation preserving and reversing respectively.
	
Using $K_k^n$ in Lemma \ref{Lem:htpygp of K2,4}, the conditions of $t$ in (\ref{lem:t orien}) follow from   Lemma \ref{Lem: Number theory x^2=1+n 8|n},  and  those in (\ref{lem:t unorien}) follow from  Lemma \ref{Lem: Number theory x2=pm 1}(\ref{NumberTheo x^2=-1}). 
\end{proof}

\begin{lemma}\label{lem: change theta to f0,n}
   If we replace the $\theta_{k}^n$ by $f_{0,n}$ in the definition of manifolds $M$ and $M'$, then Lemma \ref{lem: equi relation}  still holds.
\end{lemma}
\begin{proof}
	~~
	Assume that $f_{0,n}=\theta_{k}^n+a_k^ni_{2k-1} \xi_{2k-1}\in I_{k}^n$. Note that 
	$ni_{2k-1}\xi_{2k-1}=0$ by Lemma \ref{Lem:htpygp of K2,4}. 
Lemma \ref{lem:Mmn =M-mn} implies that $f_{m,n}\sim f_{-m-n,n}$, hence $g^n_{m,0} f_{m,n}\sim g^n_{-m-n,0} f_{-m-n,n}$ in $I_{k}^n$. 
By Proposition \ref{pro:fm,n} (\ref{condition fmn:i}),  we get the following equivalence (orientation preserving )
\begin{align*}
	&f_{0,n}+mi_{2k-1} \xi_{2k-1}\sim f_{0,n}+(-m-n)i_{2k-1} \xi_{2k-1}=f_{0,n}-mi_{2k-1} \xi_{2k-1} \\
	\Rightarrow~~ & \theta_{k}^n+(a_k^n+m)i_{2k-1} \xi_{2k-1}\sim 	\theta_{k}^n+(a_k^n-m)i_{2k-1} \xi_{2k-1}.
\end{align*}
In the following, we only prove this lemma for the case $k=4$ since the proof of that for the case $k=2$ is similar and easier. 

Lemma \ref{lem: equi relation} (\ref{lem:t orien}),   we get that for any $m\in \Z$, 
\begin{align*}
	a_4^n-m=t(a_4^n+m)+\frac{n(n-1)}{2}\varepsilon  ~\text{(mod $(120, n)$)},
\end{align*}
where $\varepsilon\in \{0,1\}$ for $8||n$, otherwise $\varepsilon=0$. 

Then from   Lemma \ref{Lem: Number theory x2=pm 1} (\ref{NumberTheo x^2=1}) (\ref{NumberTheo x^2=1 solution}),   we get that for any $m\in \Z$, 
\begin{align}
  & 4\varepsilon[1+(a_4^n+m)]\pm(a_4^n+m)\equiv (a_4^n-m) ~\text{(mod $(8,n)$)}; \label{equ:a4n+m for mod(8 n)}\\
& \pm(a_4^n+m) \!\equiv \! a_4^n-m ~\text{(mod $(3,n)$)} ~\text{and}~		\pm(a_4^n+m) \!\equiv \!a_4^n-m~\text{(mod $(5,n)$)}; \label{equ:a4n+m for mod(5 n)}.
\end{align}

By the arbitrariness of the integer $m$, from (\ref{equ:a4n+m for mod(5 n)}), it is evident that
\begin{align*}
	2a_4^n \equiv 0 ~\text{(mod $(3,n)$)},   ~~\text{and}~ 2a_4^n \equiv 0 ~\text{(mod $(5,n)$)}.
\end{align*}
If we  take $m\in \Z$ such that 
$a_4^n+m$ is odd and $(8,n)\nmid m$ in (\ref{equ:a4n+m for mod(8 n)}),  we also obtain
\begin{align*}
	2a_4^n \equiv 0 ~\text{(mod $(8,n)$)}. 
\end{align*}
Thus $f_{0,n}=\theta_{4}^n+a_{4}^ni_{7}\xi_{7}$ with  	$2a_4^n \equiv 0 ~\text{(mod $(120,n)$)}$.   So
 \begin{align*}
 	(2\iota_P)f_{0,n}=&	(2\iota_P)\theta_{4}^n+(2\iota_P)(a_{4}^ni_{7}\xi_{7})
 	\\
	=&4\theta_{4}^n	+2a_{4}^ni_{7}\xi_{7},~~\text{by \cite[Theorem (2)]{SASAO2}}\\
	=&4\theta_{4}^n\\
	=&4f_{0,n}. 
\end{align*}
Hence $f_{0,n}\in K_{4}^n$ also satisfies the equation  (\ref{equ: 2theta_k^n k=2,4}).
Thus we can take  $\theta_4^{n}=f_{0,n}$.

\end{proof}

\begin{proof}[Proof of Theorem \ref{thm:S3S4 bundle}, \ref{thm:S7S8 bundle} and Corollary \ref{Cor:Gn for k=2,4}]
	~~
	
 For any fixed integer $n$,  take  $\theta_k^{n}=f_{0,n}$. 
 
 For any $m,m'\in \Z$,  from  Proposition \ref{pro:fm,n} (\ref{condition fmn:i}), 
 \begin{align*}
 	f_{m,n}\sim f_{m',n}\in I_k^n~~\text{if and only if}~~\theta_k^{n}+mi_{2k-1}\xi_{2k-1}\sim  \theta_k^{n}+m'i_{2k-1}\xi_{2k-1} \in I_k^n. 
 \end{align*}
 Then  Part (I) of both Theorem \ref{thm:S3S4 bundle} and \ref{thm:S7S8 bundle} are obtained by  Lemma \ref{lem: equi relation} (\ref{lem:t orien}) and Lemma \ref{Lem: same equ for m m'};   Part (II) of both Theorem \ref{thm:S3S4 bundle} and \ref{thm:S7S8 bundle} are obtained by  Lemma \ref{lem: equi relation} (\ref{lem:t unorien}) and Lemma \ref{Lem: Number theory x2=pm 1} (\ref{NumberTheo x^2=-1}) (\ref{NumberTheo x^2=-1 5|n}).

\end{proof}

\section{Classification of $S^{2k-1}$-fibrations over $S^{2k}$ by attaching maps for $k=3,5,6$}
\label{sec:fibrations for k=3,5,6}

	Let $f_1=\theta_{k}^n+i_{2k-1}\xi_1$ and $f_2=\theta_{k}^n+i_{2k-1}\xi_2$ with $\xi_1,\xi_2\in \pi_{4k-2}(S^{2k-1})$.

First, we prove Theorem \ref{Thm: all odd k fibr}.
\begin{proof}[Proof of  Theorem \ref{Thm: all odd k fibr}]
	 Assume that there is an orientation reversing homotopy equivalence from $\Pn\cup_{f_1} e^{4k-1}$ to  $\Pn\cup_{f_2} e^{4k-1}$. Namely, there is $g=t\iota_P+\epsilon i_{2k-1}\eta_{2k-1}p_{2k}$ such that $g f_1=-f_2$. 
	
	From Lemma \ref{Lem: composition any k} (\ref{comp any k}), 
	\begin{align*}
		&(t\iota_P+\epsilon i_{2k-1}\eta_{2k-1} p_{2k})(\theta_{k}^n+i_{2k-1}\xi_1)=t^2\theta_{k}^n+i_{2k-1}\xi_1' ~~\text{for some}~\xi_1'\in \pi_{4k-2}(S^{2k-1}).  \\
		\text{Hence,}~& t^2\theta_{k}^n+i_{2k-1}\xi_1'=-\theta_{k}^n-i_{2k-1}\xi_2~\Rightarrow~(t^2+1)\theta_{k}^n+i_{2k-1}(\xi_1'+\xi_2)=0. 
	\end{align*}
	
	For even $k$  or $4|n$, by  Proposition \ref{Prop:htpygp of K neq2,4} (\ref{o(theat) otherwise}), we get $t^2+1\equiv 0$ (mod $2n$).  
	
	For  odd  $k$  and $2||n$, by Proposition \ref{Prop:htpygp of K neq2,4} (\ref{o(theat) odd k 2||n}),  $t^2+1\equiv 0$ (mod $4n$) or $t^2+1\equiv 2n$ (mod $4n$). These  also implies that $t^2+1\equiv 0$ (mod $2n$). 
	
	This contradicts  Lemma \ref{Lem: Number theory x2=pm 1}(\ref{NumberTheo x^2=-1}). 
\end{proof}

The following theorem from our previous work provides the classification theorem for the $S^{2k-1}$-fibration over $S^{2k}$  when $k\neq 2,4$.
\begin{theorem}[Theorem 1.6 of \cite{ZhuPanfiber}]\label{thm: x1=X2}
	Let $k\neq 2,4$ and   $X_i= \Pn\cup_{f_i}e^{4k-1}$ with $j_{\ast}(f_i)=[X_{2k},\iota_{2k-1}]_r$ ($i=1,2$) be homotopy equivalent to the total space of $S^{2k-1}$-fibration over $S^{2k}$. Then
	$X_1 \simeq X_2$ if and only if there is an integer $t$ such that  $t\Sigma^{\infty}f_1=\Sigma^{\infty}f_2$  with $t^2\equiv 1 $ (mod $2n$). 
\end{theorem}

\begin{proof}[Proof of Theorem \ref{Thm: Gkn}]\renewcommand{\qedsymbol}{}
	~~
	For $2|n$,  then
	\begin{align*}
		I_{k}^{n}=	\left\{
		\begin{array}{ll}
			\{\theta_3^{n}+bi_5\nu_5\eta_8^2~|~ b\in \Z_2\}, & \hbox{$k=3$;} \\
			\{\theta_5^{n}+i_9\bar\xi+bi_9\xi_0~|~ \bar\xi\in S_0,   b\in \Z_2\},& \hbox{$k=5$;}\\
			\{\theta_6^n\!+\!bi_{11}\zeta_{11}\!+\! c \rho_{3}^n i_{11} \alpha_{3}^{11}\!+\!e\rho_{7}^n i_{11} \alpha_{1}^{11}~|~ b\in \Z_{(8,n)},  c\in \Z_{(9,n)}, e\in \Z_{(7,n)}\},& \hbox{$k=6$.}	
		\end{array}
		\right.
	\end{align*}
	where $\xi_0=\nu_9^3+\eta_9\varepsilon_{10}+\sigma_9\eta_{16}^2$; $2n\theta_3^{n}=i_5\nu_5\eta_8^2$ and $2n\theta_5^{n}=i_9\xi_0$ for $2||n$; $2n\theta_5^{n}=0$ for $4|n$.
	
	Moreover, by \cite[Corollary 1.8]{ZhuPanfiber}, $\theta_k^{n}$ can be chosen such that 
	  \begin{align*}
	  	\Sigma^{\infty} \theta_k^n=0. 
	  \end{align*}
Hence,  from Theorem \ref{thm: x1=X2}, 
\begin{align}
	&  f_1=\theta_{k}^n+i_{2k-1}\xi_1\sim f_2=\theta_{k}^n+i_{2k-1}\xi_2   ~\text{in}~I_{k}^{n} 	\Leftrightarrow~~  ti_{2k-1}^s\Sigma^{\infty}\xi_1=i_{2k-1}^s\Sigma^{\infty}\xi_2 \label{iff f1 equ f2}
\end{align}
where  $t^2\equiv 1$ (mod $2n$) and  $i_{2k-1}^s=\Sigma^{\infty}i_{2k-1}$. 

$\bullet$  For $k=3$, by $\pi_{10}^{s}(S^5)=0$, we get 
\begin{align*}
	G_3^n=1~\text{and}~I_{3/\sim}^n=\{[\theta_{3}^n]\}.  
\end{align*}

$\bullet$  For $k=5$, by diagram (\ref{Diam: K,Ks}) and equation (\ref{equ:stable xi 0}), we get 
\begin{align*}
	G_5^n=8 ~\text{and}~ I_{5/\sim}^n=\{[\theta_{5}^n+i_9\bar\xi]~|~\bar\xi\in S_0\}.  
\end{align*}

$\bullet$  For $k=6$,  we have 
\begin{align*}
 	\pi_{22}(S^{11})/n\pi_{22}(S^{11})\xrightarrow[\cong]{\Sigma^{\infty}}	\pi^s_{22}(S^{11})/n\pi^s_{22}(S^{11}) \xrightarrow[\cong]{i_{11\ast}^s} K_{6}^{n,s}. 
\end{align*}
where the first isomorphism comes from \cite[Theorems 7.4, 13.4 and  13.9 ]{Toda}.

If we assume that $\xi_{i}=b_ii_{11}\zeta_{11}+ c_i \rho_{3}^n i_{11} \alpha_{3}^{11}+e_i\rho_{7}^n i_{11} \alpha_{1}^{11}\in \pi_{22}(S^{11}),  i=1,2$ in (\ref{iff f1 equ f2}),  then (\ref{iff f1 equ f2})  is equivalent to the following 
\begin{align*}
	tb_1 \equiv b_2 ~\text{(mod $(8,n)$)},
		tc_1 \equiv c_2 ~\text{(mod $(9,n)$)}, ~\text{and}~	te_1 \equiv e_2 ~\text{(mod $(7,n)$)}.
\end{align*}
By Lemma \ref{Lem: Number theory x2=pm 1} (\ref{NumberTheo x^2=1 solution}), that is 
\begin{align*}
b_1\equiv \pm b_2 ~\text{(mod $(8,n)$)},  c_1\equiv \pm  c_2  ~\text{(mod $(9,n)$)} ~\text{and}~ e_1\equiv \pm e_2~ \text{(mod $(7,n)$)}. 
\end{align*}
Let $U_n^3=\{0,1\}$ or $\{0,1,2,3,4\}$ according as $3||n$ or $9|n$. We get
\begin{itemize}
	\item for $2||n$,  $G_6^n=2(1+\rho_{3}^n+3\rho_{9}^n)(1+3\rho_{7}^n)$ and 
	$I_{6/\sim}^n=\{[\theta_6^n+bi_{11}\zeta_{11}+ c \rho_{3}^n i_{11} \alpha_{3}^{11}+e\rho_{7}^n i_{11} \alpha_{1}^{11}]~|~ b\in \Z_2,  c\in U_n^3, e\in\{ 0,1,2,3\}\}$; 
	\item  	for $4||n$,	$G_6^n=3(1+\rho_{3}^n+3\rho_{9}^n)(1+3\rho_{7}^n)$ and 
	$I_{6/\sim}^n=\{[\theta_6^n+bi_{11}\zeta_{11}+ c \rho_{3}^n i_{11} \alpha_{3}^{11}+e\rho_{7}^n i_{11} \alpha_{1}^{11}]~|~ b\in \{0,1,2\},  c\in U_n^3, e\in\{ 0,1,2,3\}\}$;
	\item  for $8|n$,  	$G_6^n=5(1+\rho_{3}^n+3\rho_{9}^n)(1+3\rho_{7}^n)$ and 
	$I_{6/\sim}^n=\{[\theta_6^n+bi_{11}\zeta_{11}+ c \rho_{3}^n i_{11} \alpha_{3}^{11}+e\rho_{7}^n i_{11} \alpha_{1}^{11}]~|~ b\in \{0,1,2,3,4\},  c\in U_n^3, e\in\{ 0,1,2,3\}\}$.   
\end{itemize} 

\end{proof}

 \qquad
 
\noindent
{\bf Acknowledgement.}
The first author was partially supported by National Natural Science Foundation of China (Grant No. 12571075); the second author was partially supported by National Natural Science Foundation of China (Grant No. 11971461 and  12571075).

\bibliographystyle{amsplain}

\end{document}